\documentclass{amsart}
\usepackage{amsfonts,amssymb,amsmath,amsthm,euscript}

\pdfoutput=1

\oddsidemargin \evensidemargin
\newtheorem{theorem}{Theorem}[section]
\newtheorem{proposition}{Proposition}[section]

\newtheorem{lemma}[theorem]{Lemma}

\theoremstyle{definition}
\newtheorem{definition}[theorem]{Definition}

\theoremstyle{remark}

\numberwithin{equation}{section}
\numberwithin{figure}{section}
\usepackage{color}
\usepackage{amsmath, latexsym, amssymb}
\usepackage{graphicx}
\usepackage{bbm,bm}
\usepackage{amsfonts}
\usepackage{subfigure}
\usepackage{amsthm}
\usepackage{amsmath}
\usepackage{graphics}
\usepackage{amssymb}
\usepackage{latexsym}
\usepackage{amscd}
\usepackage{color}
\usepackage{graphicx}
\usepackage{pict2e}

\newlength\cellsize 
\savebox2{%
\begin{picture}(15,15)
\put(0,0){\line(1,0){15}}
\put(0,0){\line(0,1){15}}
\put(15,0){\line(0,1){15}}
\put(0,15){\line(1,0){15}}
\end{picture}}
\newcommand\cellify[1]{\def\thearg{#1}\def\nothing{}%
\ifx\thearg\nothing
\vrule width0pt height\cellsize depth0pt\else
\hbox to 0pt{\usebox2\hss}\fi%
\vbox to 15\unitlength{
\vss
\hbox to 15\unitlength{\hss$#1$\hss}
\vss}}
\newcommand\tableau[1]{\vtop{\let\\=\cr
\setlength\baselineskip{-16000pt}
\setlength\lineskiplimit{16000pt}
\setlength\lineskip{0pt}
\halign{&\cellify{##}\cr#1\crcr}}}
\savebox3{%
\begin{picture}(15,15)
\put(0,0){\line(1,0){15}}
\put(0,0){\line(0,1){15}}
\put(15,0){\line(0,1){15}}
\put(0,15){\line(1,0){15}}
\end{picture}}
\newcommand\expath[1]{%
\hbox to 0pt{\usebox3\hss}%
\vbox to 15\unitlength{
\vss
\hbox to 15\unitlength{\hss$#1$\hss}
\vss}}

\DeclareMathOperator{\rea}{read}

\DeclareMathOperator{\qschur}{\mathcal{S}}
\DeclareMathOperator{\rowq}{\mathcal{RS}}
\DeclareMathOperator{\comts}{{\rm composition \; tableaux}}

\DeclareMathOperator{\func}{{\rm row-strict \; quasisymmetric \; Schur \; function}}
\newcommand{\shape}{\bm\lambda}
\newcommand{\setcomp}{\bm\beta}

\title{Row-strict quasisymmetric Schur functions}

\author{Sarah Mason}
\address[Mason]{
   Department of Mathematics\\
   Wake Forest University \\
   Winston-Salem, NC 27109 USA
   }
\email{masonsk@wfu.edu}

\author{Jeffrey Remmel}
\address[Remmel]{
   Department of Mathematics\\
   University of California, San Diego\\
   La Jolla, CA 92093-0112. USA\\
   }
\email{jremmel@ucsd.edu}

\begin{document}

\begin{abstract}
Haglund, Luoto, Mason, and van Willigenburg introduced  
a basis for quasisymmetric functions called the {\it quasisymmetric Schur function basis}, generated combinatorially through fillings of composition diagrams in much the same way as Schur functions are generated through reverse column-strict tableaux. We introduce   
a new basis for quasisymmetric functions called the {\it row-strict quasisymmetric Schur function basis}, generated combinatorially through fillings of composition diagrams in much the same way as Schur functions are generated through row-strict tableaux. We describe the relationship between this new basis and other known bases for quasisymmetric functions, as well as its relationship to Schur polynomials.  We obtain a refinement of the omega transform operator as a result of these relationships. 
\end{abstract}

\maketitle

\section{Introduction}

Quasisymmetric functions have emerged as a powerful tool for investigating many diverse areas such as symmetric functions \cite{BMSvW00, BTvW06}, combinatorial Hopf algebras \cite{ABS06}, discrete geometry \cite{BHvW03}, and representation theory \cite{Hiv00, KroThi99}. Quasisymmetric functions were introduced by Gessel as a source of generating functions for $P$-partitions \cite{Ges84}, although they appeared in a different format in earlier work by Stanley~\cite{Sta72}.  Gessel developed many properties of quasisymmetric functions and applied them to 
solve a number of problems in permutation enumeration. Gessel also proved that they were dual to Solomon's descent algebra.  This duality is further explored by Ehrenborg \cite{Ehr96}, Malvenuto and Reutenauer \cite{MalReu95}, and Thibon \cite{Thi01}.  

In \cite{HLMvW09}, Haglund, Luoto, Mason, and van Willigenburg  introduced  
a new basis for quasisymmetric functions called the {\it quasisymmetric Schur functions} which are generated combinatorially through fillings of composition diagrams in much the same way as Schur functions are generated through 
reverse column-strict tableaux. Each quasisymmetric Schur function 
is a positive sum of Demazure atoms.   In \cite{HLMvW09}, it 
was shown that the quasisymmetric Schur functions refine the Schur functions in a way that respects the Schur function decomposition into Gessel's fundamental quasisymmetric functions. In \cite{HLMvW10}, Haglund, Luoto, Mason,  
and van Willigenburg gave a refinement of the Littlewood-Richardson 
rule which proved that the product of a quasisymmetric Schur function and a Schur function expands positively as a sum of quasisymmetric Schur functions.  

This paper was motivated by an attempt to extend the duality between 
column-strict tableaux and row-strict tableaux to 
quasisymmetric Schur functions.\footnote{An extended abstract of the this paper appeared in \cite{MR2011}}  That is, 
let $\lambda = (\lambda_1, \lambda_2, \hdots , \lambda_k)$ be a partition of 
$n$.  The diagram associated to $\lambda$ (in English notation) consists of $k$ rows of left-justified boxes, or {\it cells}, such that the $i^{th}$ row from the top contains $\lambda_i$ cells.  A {\em reverse column-strict tableau $T$ 
of shape $\lambda$} is a filling of the cells of $\lambda$ with positive integers so that the rows are weakly decreasing and the columns are 
strictly decreasing.  A {\em reverse row-strict tableau $T$ 
of shape $\lambda$} is a filling of the cells of $\lambda$ with positive integers so that the rows are strictly decreasing and the columns are 
weakly  decreasing.  Let $\mathcal{RCS}_\lambda$ (resp. $\mathcal{RRS}_\lambda$) 
denote the set of all reverse column strict tableaux (resp. reverse row 
strict tableaux) of shape $\lambda$. 
If $T$ is a reverse column-strict tableau or 
a reverse row-strict tableau, we let $T(i,j)$ be the element in 
the cell which is in the $i$-th row of $T$, reading from top to bottom, 
and the $j$-th column of $T$, reading from left to right, and 
we let the weight, $x^T$,  of $T$ be defined as 
$\displaystyle{x^T = \prod_{(i,j) \in \lambda} x_{T(i,j)}}$. Then the 
Schur function $s_\lambda(x_1,x_2, \ldots)$ is defined as  
\begin{equation}\label{defschur1}
s_\lambda(x_1,x_2, \ldots) = \sum_{T \in \mathcal{RCS}_\lambda} x^T.
\end{equation}
If $T$ is a reverse column-strict tableau or 
a reverse row-strict tableau of shape $\lambda$, we define 
the conjugate, $T^\prime$, of $T$ to be the filled diagram of shape $\lambda^\prime$ which 
results by reflecting the cells of $T$ across the main diagonal. Clearly 
$T$ is a reverse column-strict tableaux if and only if 
$T^\prime$ is a reverse-row strict tableaux.  Thus 
 \begin{equation}\label{defschur2}
s_{\lambda^\prime}(x_1,x_2, \ldots) = \sum_{T \in \mathcal{RRS}_\lambda} x^T
\end{equation}
where $\lambda'$ is the transpose of the partition $\lambda$, often referred to as conjugate partition~\cite{Hal59, Hal88, Sta71}.
Moreover, if $\omega$ is the algebra isomorphism defined 
on the ring of symmetric functions $\Lambda$ so that 
$\omega(h_n) = e_n$, where $h_n= h_n(x_1,x_2, \ldots)$ is $n$-th homogeneous symmetric function 
and $e_n= e_n(x_1,x_2, \ldots)$ is $n$-th elementary symmetric function, then 
it is well known that   
\begin{equation}\label{omegaonschur}
\omega(s_\lambda(x_1,x_2, \ldots)) = 
s_{\lambda^\prime}(x_1,x_2, \ldots).
\end{equation}

The question that motivated this paper is whether we 
can find a duality like that expressed in 
(\ref{defschur1}), (\ref{defschur2}), and (\ref{omegaonschur}) for quasisymmetric Schur functions. 
In this paper, we introduce a new basis for 
quasisymmetric functions called the {\em row-strict quasisymmetric 
Schur functions}, which are generated combinatorially through fillings of composition diagrams in much the same way as Schur functions are generated through 
reverse row-strict tableaux. However, 
the process of conjugation becomes less transparent in the quasisymmetric setting since bases for quasisymmetric functions are typically indexed by compositions instead of partitions.  That is, it is not enough to simply reflect across the main diagonal since this does not necessarily produce a left-justified diagram.  In fact, the number of compositions which are rearrangements of  a given partition is generally not equal to the number of compositions which are rearrangements of  its transpose so that any relationship between these two collections must necessarily be more complex than a simple bijection.  There is a refinement of $\omega$ transformation which is 
defined on the space of quasisymmetric functions and we shall 
use this refinement to better understand the relationship between the compositions rearranging a partition and those rearranging its conjugate.

The outline of this paper is as follows. 
In Section {\ref{classical}}, we shall briefly review the background 
on symmetric and quasisymmetric functions that we shall need.   
In Section {\ref{dualbasis}}, we shall recall the definition of the quasisymmetric Schur function $\qschur_{\alpha}$ and 
we shall define the  {\it row-strict quasisymmetric Schur functions} $\rowq_{\alpha}$. We shall also describe their relationship to quasisymmetric Schur functions and Schur functions.  In Section \ref{properties}, we shall 
give combinatorial interpretations to the coefficients that 
arise in the expansion of $\rowq_{\alpha}$ in terms of both 
the monomial quasisymmetric functions and the fundamental quasisymmetric 
functions. This will allow us to prove that the row-strict quasisymmetric 
functions are a basis for the quasisymmetric functions. In Section \ref{omega},
we show that the $\omega$ involution on quasisymmetric functions is the involution which interpolates between the quasisymmetric Schur functions and the row-strict quasisymmetric Schur functions. Finally in Section \ref{Dual}, 
we shall define an extension of the dual Schensted insertion procedure 
that is appropriate for row-strict quasisymmetric functions and 
prove some its fundamental properties. This extension 
is used by Ferreira~\cite{Fer11} to prove a refinement of the Littlewood-Richardson 
rule to give a combinatorial interpretation of the product 
of a Schur function times a row-strict quasisymmetric Schur function 
as a positive sum of row-strict quasisymmetric functions.

\section{Symmetric and quasisymmetric functions}{\label{classical}}

A {\it symmetric function} is a bounded degree formal power series $f(x) \in \mathbb{Q}[[x_1,x_2, \hdots ]]$ such that $f(x)$ is fixed under the action of the symmetric group; that is, $\sigma(f(x)) = f(x)$ for all $\sigma \in \mathfrak{S}_{\infty}$ where $\sigma(f(x_1, x_2, \hdots))=f(x_{\sigma(1)}, x_{\sigma(2)}, \hdots)$.  We let $\Lambda$ denote the ring of symmetric functions 
and $\Lambda_n$ denote the space of homogeneous symmetric 
functions of degree $n$ so that $\Lambda = \oplus_{n \geq 0} \Lambda_n$.   

A {\it partition} of $n$ is a weakly decreasing sequence $\lambda=(\lambda_1, \lambda_2, \hdots \lambda_k)$ of positive integers which sum to $n$.  We write $| \lambda | = n$ and let $l(\lambda)=k$ be the {\it length} of $\lambda$. Given a partition $\lambda = (\lambda_1, \ldots, \lambda_k)$ of $n$,  we say that a reverse column-strict tableau $T$ is a 
{\em standard reverse column-strict tableau} if each of 
the numbers $1,2, \ldots, n$ appear exactly once in $T$. 
Standard reverse row-strict tableaux are defined similarly.
Note that the set of standard reverse row-strict tableaux is the same as the set of standard reverse column-strict tableaux, so we simply use the term {\it reverse standard Young tableau} when the tableau is standard.
A reverse column-strict tableau can be converted to a standard reverse column-strict tableau by a procedure known as {\it standardization}.  Let $T$ be an arbitrary reverse column-strict tableau such that $x^T = x_1^{a_1} x_2^{a_2} \hdots x_k^{a_k}$.  First replace (from right to left) the $a_1$ $1$s in $T$ with the numbers $1,2, \hdots, a_1$.  Then replace the $a_2$ $2$s with the numbers $a_1+1, \hdots, a_1+a_2$, and so on.  The resulting diagram is a standard reverse column-strict tableau, called the {\it standardization $std(T)$ of $T$}.
The standardization of a reverse row-strict tableau is defined analogously to that of a reverse column-strict tableau but with the entries replaced from bottom to top rather than from right to left.   

A {\it composition} $\alpha \vDash n$ of $n$ is a sequence of positive integers which sum to $n$.  Each composition $\alpha = (\alpha_1, \alpha_2, \hdots , \alpha_k)$ is associated to the subset of $[n-1]$ given by $S(\alpha) = \{ \alpha_1, \alpha_2 - \alpha_1, \hdots, n-\alpha_k \}$.  Note that this is an invertible procedure. That is, if $P=\{s_1, s_2, \hdots , s_k \}$ is an arbitrary subset of $[n-1]$, then the composition $\setcomp(P) = (s_1, s_1+s_2, \hdots, s_1+s_2 + \hdots s_k)$ is precisely the composition such that $S(\setcomp(P))=P$.  We will make use of the {\it refinement order $\preceq$} on compositions which states that $\alpha \preceq \beta$ if and only if $\beta$ is obtained from $\alpha$ by summing some of the parts of $\alpha$.  For example, $(3,1,1,2,1) \preceq (3,2,3)$ but $(3,1,2,1,1)$ and $(3,2,3)$ are incomparable under the refinement ordering. 
If $\alpha$ is a composition, let $\shape(\alpha)$ denote the partition obtained by arranging the parts of $\alpha$ in weakly decreasing order.  We say that $\alpha$ is a {\it rearrangement} of the partition $\shape(\alpha)$.
For example, $\shape(1,3,2,1) = (3,2,1,1)$. If $\lambda = 
(\lambda_1 \geq \cdots \geq \lambda_n)$ and $\mu = (\mu_1 \geq \cdots \geq 
\mu_n)$ are partitions of $n$, then we write $\lambda > \mu$ if 
$\lambda$ dominates $\mu$, i.e. if for all $1 \leq k \leq n$, we have \
$\sum_{i=1}^k \lambda_i \geq \sum_{i=1}^k \mu_i$. 

We also need the notion of {\it complementary compositions}.  The {\it complement $\tilde{\beta}$} to a composition $\setcomp(S)$ arising from a subset $S \subseteq [n-1]$ is the composition obtained from the subset $S^c \subseteq [n-1]$.  For example, the composition $\beta = (1,4,2)$ arising from the subset $S=\{1,5 \} \subseteq [7]$ has complement $\tilde{\beta}=(2,1,1,2,1)$ arising from the subset $S^c=\{2,3,4,6,7\}$.

A {\it quasisymmetric function} is a bounded degree formal power series $f(x) \in \mathbb{Q}[[x_1, x_2, \hdots ]]$ such that for all compositions $\alpha=(\alpha_1, \alpha_2, \hdots, \alpha_k)$, the coefficient of $\prod x_i^{\alpha_i}$ is equal to the coefficient of $\prod x_{i_j}^{\alpha_i}$ for all $i_1 < i_2 < \hdots < i_k$.  We let $Qsym$ denote the ring of quasisymmetric functions 
and $Qsym_n$ denote the space of homogeneous quasisymmetric functions 
of degree $n$ so that $\displaystyle{Qsym = \oplus_{n \geq 0} Qsym_n}$.

A natural basis for $Qsym_n$  is the {\it monomial quasisymmetric basis}, given by the set of all $M_\alpha$ such that  $\alpha \vDash n$ where  
$$M_{\alpha} = \sum_{i_1 < i_2 < \cdots < i_k} x_{i_1}^{\alpha_1} x_{i_2}^{\alpha_2} \cdots x_{i_k}^{\alpha_k}.$$  
Gessel's {\it fundamental basis for quasisymmetric functions}~\cite{Ges84} can be expressed by $$F_{\alpha} = \sum_{\beta \preceq \alpha} M_{\beta},$$ where $\beta \preceq \alpha$ means that $\beta$ is a refinement of $\alpha$.

The {\it descent set} $D(T)$ of a standard tableau is the set of all positive integers $i$ such that $i+1$ appears in a column weakly to the right of the column containing $i$.  The following theorem describes the way a Schur function can be expressed as a positive sum of fundamental quasisymmetric functions.

\begin{theorem}{\cite{Ges84}}{\label{schur2fund}}
 $$s_{\lambda} = \sum_{\beta} d_{\alpha \beta} F_{\beta},$$ where $d_{\alpha \beta}$ is equal to the number of standard reverse column-strict tableaux $T$ of shape $\alpha$ and descent set $D(T)$ such that $\beta(D(T))=\beta$.
\end{theorem}

\begin{proof}
We claim that each fundamental quasisymmetric function $F_{\alpha}$ can be thought of as the generating function for all reverse column-strict tableaux which standardize to a fixed standard reverse column-strict tableau.
To see this, consider an arbitrary standard reverse column-strict tableau $T$ with descent set $\{b_1, b_2, \hdots , b_k \}$.  Then the cells of $T$ containing the entries $\{b_{i-1}+1, b_{i-1}+2, \hdots , b_{i}\}$ appear in horizontal strips, called {\it row strips}, so that no two such cells appear in the same column in $T$.  Each reverse column-strict tableau which standardizes to $T$ can similarly be decomposed into the same collection of row strips, so that each destandardization of $T$ corresponds to a coarsening of the weight of $T$.  The coarsest such tableau is the reverse column-strict tableau obtained from $T$ by replacing the entries $\{ b_{i-1}+1, b_{i-1}+2, \hdots , b_i\}$ with the entry $i$.  Therefore exactly the monomial quasisymmetric functions indexed by refinements of the composition $\setcomp(\{b_1, b_2, \hdots , b_k\})$ occur as the weights of the reverse column-strict tableaux which standardize to $T$ and the claim is true.

The claim implies that the fundamental quasisymmetric functions appearing in the Schur function $s_{\lambda}$ are precisely those corresponding to the descent sets of the standard reverse column-strict tableaux of shape $\lambda$, as desired.
\end{proof}

An extension of the classical $\omega$ transformation on 
symmetric functions (defined in the introduction to the space of quasisymmetric functions) appears in the work of Ehrenborg~\cite{Ehr96}, Gessel~\cite{Ges84}, and Malvenuto-Reutenauer~\cite{MalReu95}.  One can define this endomorphism on the fundamental quasisymmetric functions by $\omega(F_{\alpha}) = F_{rev(\tilde{\alpha})}$, where $rev(\tilde{\alpha})$ is the composition obtained by reversing the order of the entries in $\tilde{\alpha}$.  Then $\omega$ is an automorphism of the algebra of quasisymmetric functions whose restriction to the space of symmetric function equals the classical $\omega$ transformation. 

\section{Quasisymmetric Schur functions and row-strict quasisymmetric 
Schur functions}{\label{dualbasis}}

Let $\alpha=(\alpha_1, \alpha_2, \hdots , \alpha_l)$ be a composition of $n$.  The diagram associated to $\alpha$ consists of $l$ rows of left-justified boxes, or {\it cells}, such that the $i^{th}$ row from the top contains $\alpha_i$ cells, as in the English notation.  Given a composition diagram $\alpha=(\alpha_1, \alpha_2, \hdots , \alpha_l)$ with largest part $m$, a {\it column-strict composition tableau} (CSCT), $F$ is a filling of the cells of $\alpha$ with positive integers such that\begin{enumerate}
\item the entries of $F$ weakly decrease in each row 
when read from left to right,
\item the entries in the leftmost column of $F$ strictly increase when read from top to bottom, 
\item and $F$ satisfies the column-strict triple rule. 
\end{enumerate}
Here we say that $F$ satisfies the {\em column-strict triple rule} if when we 
supplement $F$ by adding enough cells with zero-valued entries to the end of each row so that the resulting supplemented tableau, $\hat{F}$, is of rectangular shape $l \times m$, then for $1 \le i < j \le l, \; \; 2 \le k \le m$
$$(\hat{F}(j,k) \not= 0 \; {\rm and} \; \hat{F}(j,k) \ge \hat{F}(i,k)) \Rightarrow \hat{F}(j,k) > \hat{F}(i,k-1)$$
where $\hat{F}(i,j)$ denotes the entry of $\hat{F}$ that lies 
in the cell in the $i$-th row and $j$-th column. 

Define the type of a column-strict composition tableau $F$ to be the weak composition $w(F)=(w_1(F),w_2(F), \hdots)$ where $w_i(F)=$ the number of times $i$ appears in $F$.  The weight of $F$ is $$x^F=\prod_{i} x_i^{w_i(F)}.$$ A CSCT $F$ with $n$ cells is {\it standard} if $x^F = \prod_{i=1}^n x_i$.
If $T$ is a standard {\it CSCT}, then we define its descent set, 
$D(T)$ to be  the set of all $i$ such that $i+1$ appears in a column weakly to the right of the column containing $i$. 

Haglund, Luoto, Mason, and van Willigenburg \cite{HLMvW09} 
defined the quasisymmetric Schur function $\qschur_{\alpha}$ by 
\begin{equation}\label{aqschur}
 \qschur_{\alpha} = \sum_{F} x^F
\end{equation}
where the sum runs over all column-strict composition tableaux of 
shape $\alpha$.  They show that $\qschur_{\alpha}$ as $\alpha$ ranges 
over all compositions of $n$ is a basis for the space 
$Qsym_n$. They also 
show that for any partition $\lambda$ of $n$, 
\begin{equation}\label{colstdecomp}
s_\lambda = \sum_{\alpha: \shape(\alpha) = \lambda} \qschur_{\alpha}
\end{equation}
They provide a combinatorial description of the expansion  
of $\qschur_{\alpha}$ in terms of Gessel's fundamental 
quasisymmetric function basis in the following proposition. 

\begin{proposition}~\cite[Proposition 6.2]{HLMvW09}{\label{qschur2fund}}
Let $\alpha, \beta$ be compositions.  Then $$\qschur_{\alpha} = \sum_{\beta} d_{\alpha \beta} F_{\beta}$$  where $d_{\alpha \beta} = $ the number of standard column-strict composition tableaux $T$ of shape $\alpha$ and $\bm\beta(D(T)) = \beta$.
\end{proposition}

Next we define our row-strict version of the quasisymmetric Schur 
functions. Let $\alpha=(\alpha_1, \alpha_2, \hdots , \alpha_l)$ be a composition of $n$.  Given a composition diagram $\alpha$  with largest part $m$, we define a {\it row-strict composition tableau} (RSCT), $F$, to be a filling of the cells of $\alpha$ with positive integers such that
\begin{enumerate}
\item the entries of $F$ strictly decrease in each row when read from left to right,
\item the entries in the leftmost column of $F$ weakly increase when read from top to bottom, 
\item and $F$ satisfies the row-strict triple rule. 
\end{enumerate}
Here we say that $F$ satisfies the {\em row-strict triple rule} if 
when we supplement $F$ by adding enough cells with zero-valued entries to the end of each row so that the resulting supplemented tableau, $\hat{F}$, is of rectangular shape $l \times m$, then for $1 \le i < j \le l, \; \; 2 \le k \le m$
$$(\hat{F}(j,k) > \hat{F}(i,k)) \Rightarrow \hat{F}(j,k) \ge \hat{F}(i,k-1).$$

\begin{figure}
$$F=\tableau{2 & 1 \\ 2 \\ 3 & 2 \\ 3} \; , \qquad \hat{F}=\tableau{2 & 1 \\ 2 & 0 \\ 3 & 2 \\ 3 & 0} \; , \qquad x^F=x_1 x_2^3 x_3^2$$
\label{fig:row}
\caption{An RSCT $F$ of shape $(2,1,2,1)$ and weight $x_1 x_2^3 x_3^2$}
\end{figure}

This mirrors the definition of a composition tableau given in \cite{HLMvW09} and presented above, interchanging the roles of weak and strict.  Continuing this analogy, define the type of a row-strict composition tableau $F$ to be the weak composition $w(F)=(w_1(F),w_2(F), \hdots)$ where $w_i(F)=$ the number of times $i$ appears in $F$.  The weight associated with $F$ is $$x^F=\prod_{i} x_i^{w_i(F)}.$$  Also, a RSCT $F$ with $n$ cells is {\it standard} if $x^F = \prod_{i=1}^n x_i$.  See Figure {\ref{fig:row}} for an example of a RSCT and its weight.

\begin{definition}
Let $\alpha$ be a composition.  Then the $\func \rowq_{\alpha}$ is given by $$\rowq_{\alpha} = \sum_T x^T$$ where the sum is over all RSCT's $T$ of shape $\alpha$.  See Figure {\ref{polynomial}} for an example.
\end{definition}

\begin{figure}
$$\tableau{2&1 \\ 2 \\ 3&2 \\ 3} \qquad \tableau{2&1 \\ 2 \\ 3&2 \\ 4} \qquad \tableau{2&1 \\ 2 \\ 4&3 \\ 4} \qquad \tableau{2&1 \\ 2 \\ 4&2 \\ 4} \qquad \tableau{2&1 \\ 3 \\ 4&3 \\ 4} \qquad \tableau{3&1 \\ 3 \\ 4&3 \\ 4} \qquad \tableau{3&2 \\ 3 \\ 4&3 \\ 4}$$
\caption{The $\func \rowq_{(2,1,2,1)}(x_1,x_2,x_3,x_4)=$ $x_1x_2^3x_3^2 + x_1x_2^3x_3x_4 + x_1x_2^2x_3x_4^2 + x_1x_2^3x_4^2 + x_1x_2x_3^2x_4^2 + x_1x_3^3x_4^2 + x_2x_3^3x_4^2$}
\label{polynomial}
\end{figure}

\begin{proposition}
If $\alpha$ is an arbitrary composition, then $\rowq_{\alpha}$ is a quasisymmetric polynomial.
\end{proposition}

\begin{proof}
Consider an arbitrary monomial $x_1^{a_1} x_2^{a_2} \hdots x_k^{a_k}$ appearing in $\rowq_{\alpha}$ with coefficient $c$.  We must prove that $x_{i_1}^{a_1} x_{i_2}^{a_2} \hdots x_{i_k}^{a_k}$, where $i_1 < i_2 < \hdots < i_k$, also appears in $\rowq_{\alpha}$ with coefficient $c$.  Each of the $c$ appearances of the monomial $x_1^{a_1} x_2^{a_2} \hdots x_k^{a_k}$ in $\rowq_{\alpha}$ corresponds to an RSCT $T$.  Renumber the entries of $T$ so that each appearance of the entry $j$ becomes $i_j$ for each $j$ from $1$ to $k$.  The resulting figure satisfies the RSCT conditions since the relative order of the cell entries is not altered by the renumbering procedure.  Since the renumbering procedure is a bijection between RSCTs of weight $x_1^{a_1} x_2^{a_2} \hdots x_k^{a_k}$ and RSCTs of weight $x_{i_1}^{a_1} x_{i_2}^{a_2} \hdots x_{i_k}^{a_k}$, the coefficient of $x_1^{a_1} x_2^{a_2} \hdots x_k^{a_k}$ is equal to the coefficient of $x_{i_1}^{a_1} x_{i_2}^{a_2} \hdots x_{i_k}^{a_k}$ for any choice of $i_1 < i_2 < \hdots < i_k$ and therefore the polynomial is a quasisymmetric polynomial.
\end{proof}

We shall see in Section \ref{transitions} that the row-strict quasisymmetric Schur functions form a basis for quasisymmetric functions.  The row-strict quasisymmetric Schur functions in $n$ variables form a basis for $Qsym_n$ which is different 
from the basis of $Qysm_n$ formed by the quasisymmetric Schur functions even though they can be described through a similar process.  
For example, 
the transition matrix between the two bases is given in Figure {\ref{qs2dqs}} for $n=4$ where each row gives the expansion of $\rowq_{\alpha}$ in terms 
of the quasisymmetric Schur functions. This illustrates that their relationship is fairly complex.

\begin{figure}
\begin{center}
\begin{tabular}{l || c | c | c | c | c | c | c | c |}
 $\rowq_{\alpha} \backslash \qschur_{\alpha}$  & 4 & 31 & 13 & 22 & 211 & 121 & 112 & 1111 \\
\hline
\hline
4 & 0 & 0 & 0 & 0 & 0 & 0 & 0 & 1 \\
\hline
31 & 0 & 0 & 0 & 0 & 0 & 0 & 1 & 0 \\
\hline
13 & 0 & 0 & 0 & 0 & 1 & 1 & 0 & 0 \\
\hline
22 & 0 & 0 & 0 & 1 & 0 & 0 & 0 & 0 \\
\hline
211 & 0 & 0 & 1 & -1 & 0 & 1 & 0 & 0 \\
\hline
121 & 0 & 0 & 0 & 1 & 0 & -1 & 0 & 0 \\
\hline
112 & 0 & 1 & 0 & 0 & 0 & 0 & 0 & 0 \\
\hline
1111 & 1 & 0 & 0 & 0 & 0 & 0 & 0 & 0 \\
\hline
\end{tabular}
\end{center}
\caption{The transition matrix from row-strict quasisymmetric Schur functions to quasisymmetric Schur functions}{\label{qs2dqs}}
\end{figure}

\subsection{Decomposing a Schur function into row-strict quasisymmetric Schur functions}

Recall that a {\it reverse row-strict tableau} $T$ of partition shape $\lambda$ is a filling of the cells of the diagram of $\lambda$ with positive integers such that the entries in each column weakly decrease from top to bottom and the entries in each row strictly decrease from left to right.  Row-strict composition tableaux are related to reverse row-strict tableaux by a simple bijection, which is analogous to the bijection between column-strict composition tableaux and reverse column-strict tableaux \cite{Mas08}.  

The following map sends a reverse row-strict tableau $T$ to a RSCT $\rho(T)=F$.  We describe it algorithmically.
\begin{enumerate}
\item Begin with the entries in the leftmost column of $T$ and place them into the first column of $F$ in weakly increasing order from top to bottom.
\item After the first $k-1$ columns of $T$ have been placed into $F$, place the entries from the $k^{th}$ column of $T$ into $F$, beginning with the largest.  Place each entry $e$ into the cell $(i,k)$ in the highest row $i$ such that $(i,k)$ does not already contain an entry from $T$ and the entry $(i,k-1)$ is strictly greater than $e$.
\end{enumerate}

\begin{figure}
$$T=\tableau{7&6&5&4&2 \\ 7&5&3&1 \\ 6&4&2&1 \\ 2} \; \; \; \qquad \rho(T)=\tableau{2 \\ 6&5&3&1 \\ 7&6&5&4&2 \\ 7&4&2&1}$$
\caption{The bijection $\rho$ applied to a reverse row-strict tableau $T$}
{\label{ex-map}}
\end{figure}

See Figure {\ref{ex-map}} for an example of the map $\rho$ from a reverse row-strict tableau $T$ to a RSCT $\rho(T)=F$.

\begin{lemma}{\label{RCTtoTAB}}
The map $\rho$ is a weight preserving bijection between the set of reverse row-strict tableaux of shape $\lambda$ and the set of RSCT's of shape $\alpha$ where $\shape(\alpha)=\lambda$.
\end{lemma}

\begin{proof}
Given an arbitrary reverse row-strict tableau $T$, we must first prove that the figure $\rho(T)=F$  is indeed an RSCT.  The row entries satisfy the first RSCT condition by construction and the leftmost column satisfies the second RSCT condition by construction.  Therefore it is enough to show that the 
row-strict triple condition is satisfied.

Consider the supplemented tableau $\hat{F}$ and select two entries $\hat{F}(i,k), \hat{F}(j,k)$ in the same column of $\hat{F}$ such that $\hat{F}(j,k) \not= 0$ and $\hat{F}(j,k) > \hat{F}(i,k)$.  The entry $\hat{F}(j,k)$ is inserted into $F$ before $\hat{F}(i,k)$ since $\hat{F}(j,k) > \hat{F}(i,k)$.  This means that the cell $(i,k)$ of $F$ is empty during the insertion of $\hat{F}(j,k)$.  The entry $\hat{F}(i,k-1)$ must therefore be less than or equal to $\hat{F}(j,k)$, for otherwise $\hat{F}(j,k)$ would be placed into cell $(i,k)$.

The shape of $\rho(F)$ is a rearrangement of $\lambda$ by construction, since the column sets of $\rho(T)$ are the same as the column sets of $T$.  Therefore the map $\rho$ sends a reverse row-strict tableau of shape $\lambda$ to an RSCT of shape $\alpha$ where $\shape(\alpha) = \lambda$.

To see that the map is a bijection, we describe its inverse.  Given an RSCT $F$, we create a reverse row-strict tableau $\rho^{-1}(F)=T$ by rearranging the column entries into weakly decreasing order from top to bottom.  We must prove that the resulting figure is indeed a reverse row-strict tableau.  The columns of $T$ are weakly decreasing by construction, so we must prove that the rows are strictly decreasing.

Consider the $i^{th}$ largest entry in the $j^{th}$ column of $T$.  This entry appears in $F$ immediately to the right of a strictly greater entry, as do all of the $i-1$ entries greater than or equal to this entry.  Therefore there are at least $i$ strictly greater entries in column $j-1$ of $T$ and therefore the $i^{th}$ largest entry in the $j^{th}$ column of $T$ appears immediately to the right of a strictly greater entry.  Since this is true for all entries in all columns of $T$, the rows of $T$ are strictly decreasing when read from left to right and therefore $T$ is a reverse row-strict tableau.
\end{proof}

Lemma \ref{RCTtoTAB} implies that each Schur function $s_{\lambda}$ decomposes into a positive sum of row-strict quasisymmetric Schur functions indexed by compositions that rearrange the transpose of $\lambda$.  That is, $$s_{\lambda} = \sum_{\alpha: \shape(\alpha) = \lambda'} \rowq_{\alpha}.$$

An arbitrary Schur function can therefore be decomposed into either a sum of quasisymmetric Schur functions~\cite{HLMvW09} or a sum of row-strict quasisymmetric Schur functions.

\section{Properties of the row-srict quasisymmetric Schur functions}{\label{properties}}

In order to develop several fundamental properties of the row-strict quasisymmetric Schur functions, we need to understand the behavior of the map $\rho$ which interpolates between row-strict composition tableaux and reverse row-strict 
tableaux.

\subsection{Properties of the map $\rho$}

Recall that a reverse row-strict tableau can be converted to a standard Young tableau by the standardization procedure described in Section {\ref{classical}}.  The standardization of a row-strict composition tableau  $F$ is defined similarly to that of a reverse row-strict tableau, where the entries are replaced beginning with the leftmost column and moving left to right, replacing entries within a column from bottom to top except for the entries in the leftmost column, which are replaced from top to bottom.  The resulting standard composition tableau is called the {\it standardization} of $F$ and denoted $st(F)$.  See Figure {\ref{fig:standard}} for an example.  Note that this procedure does not alter the shape of $F$.

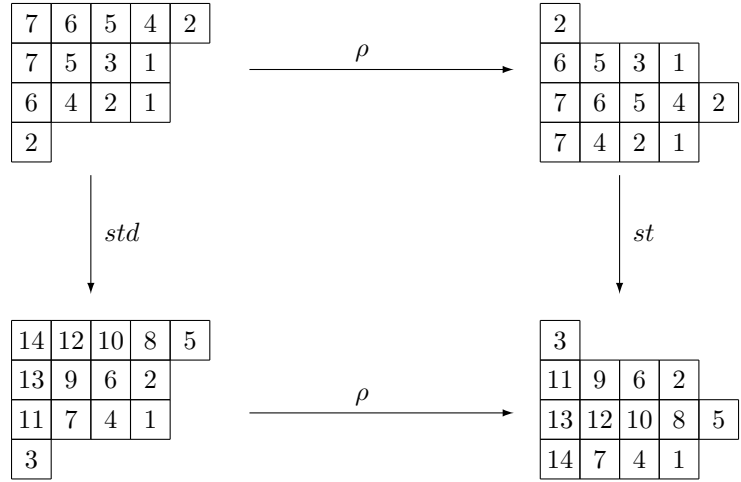
\begin{figure}{\label{fig:standard}}
\begin{center}
\begin{picture}(250,200)
\put(0,180){$\tableau{7&6&5&4&2 \\ 7&5&3&1 \\ 6&4&2&1 \\ 2}$}
\put(200,180){$\tableau{2 \\ 6&5&3&1 \\ 7&6&5&4&2 \\ 7&4&2&1}$}
\put(0,60){$\tableau{14&12&10&8&5 \\ 13&9&6&2 \\ 11&7&4&1 \\ 3}$}
\put(200,60){$\tableau{3 \\ 11&9&6&2 \\ 13&12&10&8&5 \\ 14&7&4&1 }$}
\put(30,130){\vector(0,-1){45}}
\put(230,130){\vector(0,-1){45}}
\put(90,170){\vector(1,0){100}}
\put(90,40){\vector(1,0){100}}
\put(130,175){$\rho$}
\put(130,45){$\rho$}
\put(35,105){$std$}
\put(235,105){$st$}
\end{picture}
\end{center}
\caption{The map $\rho$ commutes with standardization.}
\end{figure}

\begin{proposition}{\label{prop:standard}}
Standardization commutes with the map $\rho$ in the sense that if $T$ is an arbitrary reverse row-strict tableau then $st(\rho(T))=\rho(std(T))$.
\end{proposition}

\begin{proof}
The map $\rho$ does not alter the set of column entries.  That is, the set of entries in column $c$ of $T$ are equal to the set of entries in column $c$ of $\rho(T)$.  Therefore it is enough to show that the set of entries in an arbitrary column of $std(T)$ is equal to the set of entries in the corresponding column of $st(\rho(T))$.

Consider two equal entries in a reverse row-strict tableau.  The lower entry must appear weakly left of the higher entry since the row entries are strictly decreasing left to right and the column entries are weakly decreasing top to bottom.  Therefore during the standardization process the entries are replaced beginning with the leftmost column and moving left to right.  Therefore the set of column entries in an arbitrary column of $st(T)$ is equal to the set of entries in the corresponding column of $st(\rho(T))$ and thus standardization commutes.
\end{proof}

\subsection{Transitions to classical quasisymmetric function bases}{\label{transitions}}

Each monomial in a row-strict quasisymmetric Schur function corresponds to a row-strict composition tableau whose weight corresponds to the non-zero exponents in the monomial.  Consider the monomial $x_1^{\beta_1} x_2^{\beta_2} \cdots x_k^{\beta_k}$ with exponent composition $\beta = (\beta_1, \beta_2, \hdots , \beta_k)$.  The coefficient of this monomial in $\rowq_{\alpha}$ is equal to the number of row-strict composition tableaux of shape $\alpha$ and weight $\beta$.  By 
shifting the entries in a row-strict composition tableau appropriately, 
it is easy  to see that $$\rowq_{\alpha} = \sum_{\beta} K(\alpha, \beta)^r M_{\beta},$$ where $K(\alpha, \beta)^r$ is the number of row-strict composition tableaux of shape $\alpha$ and weight $\beta$.

  Given a standard row-strict composition tableau $F$, its {\it transpose descent set $D'(F)$} is the set of all $i$ such that $i+1$ appears in a column strictly to the left of the column containing $i$.  Since the map $\rho$ preserves the column sets of the diagram, the transpose descent set of a reverse row-strict standard Young tableau $T$ (defined analogously) is equal to the transpose descent set of $\rho(T)$.  To each transpose descent set $S=\{s_1,s_2, \hdots s_k\}$ one associates a unique composition $\bm\beta(S)=(s_1, s_2-s_1, \hdots , s_k-s_{k-1},n-s_k)$ whose successive parts are equal to the differences between consecutive elements of the set $S$ and whose last part is given by the difference of $n$ and the last element $s_k$ in the set $S$.

Next we prove the analogue of Proposition \ref{qschur2fund} for 
row-strict quasisymmetric Schur functions. 

\begin{proposition}{\label{rowq2fund}}
Let $\alpha, \beta$ be compositions.  Then $$\rowq_{\alpha} = \sum_{\beta} d_{\alpha \beta} F_{\beta}$$ where $d_{\alpha \beta}$ is equal to the number of standard row-strict composition tableaux $F$ of shape $\alpha$ and $\beta(D'(F)) = \beta$.
\end{proposition}

\begin{proof}
We claim that
\begin{equation}{\label{Fdecomp}}
\rowq_{\alpha} = \sum_T F_{\beta(D'(T))},
\end{equation} where the sum is over all row-strict standard $\comts$ $T$ of shape $\alpha$.  The proposition follows easily from this claim since $d_{\alpha \beta}$ counts the number of times $F_{\beta}$ appears on the right hand side of \ref{Fdecomp}.

To prove the claim, first note that $\rowq_{\alpha}$ is a sum of monomials arising from row-strict composition tableaux of shape $\alpha$.  Each such filling $F$ maps to a reverse row-strict tableau via the bijection $\rho^{-1}$, which then maps to a reverse column-strict  tableau under the transpose bijection.  Consider an arbitrary standard reverse column-strict tableau $T$ appearing in the image.   Since standardization does not alter the shape of the composition diagram and commutes with $\rho^{-1}$ by Proposition {\ref{prop:standard}}, all of the reverse 
column-strict tableaux which standardize to $T$ appear in the image.  This implies that $\rowq_{\alpha}$ is a positive sum of fundamental quasisymmetric functions since each fundamental quasisymmetric function is a sum of the monomials arising from a standard reverse row-strict tableau $T$ together with the reverse row-strict tableaux which standardize to $T$. 

To determine the fundamental quasisymmetric functions appearing in $\rowq_{\alpha}$, consider an arbitrary standard reverse row-strict tableau $T$ appearing in the image of the row-strict composition tableaux of shape $\alpha$ under the composition of the maps $\rho^{-1}$ and transpose.  The descent set $D(T)$ of this tableau determines the fundamental quasisymmetric function containing the monomial associated to $T$ by Theorem {\ref{schur2fund}}.  Recall that $D(T)$ is the set of all $i$ such that $i+1$ appears in a column weakly to the right of the column containing $i$, and in fact this means that $i+1$ must appear in a row strictly above $i$ since the row and column entries are decreasing.  This implies that in the transpose $T'$ of $T$ the entry $i+1$ must appear in a column strictly to the left of $i$ and hence $i$ is contained in the transpose descent set.  Similarly, if $i$ is not in the descent set of $T$ then $i+1$ appears in a row weakly below $i$ and therefore $i$ is not in the transpose descent set of $T'$.  The map $\rho$ preserves transpose descent sets since it preserves the entries in a given column.  Therefore the fundamental quasisymmetric functions appearing in $\rowq_{\alpha}$ are precisely those indexed by the compositions corresponding to the transpose descent sets of the standard row-strict composition tableaux.
\end{proof}

\begin{theorem}
The set $\{\rowq_{\alpha}(x_1, \ldots, x_k) | \alpha \vDash n \; {\rm and} \; k \ge n \}$ forms a $\mathbb{Z}$-basis for $QSym_n(x_1, \ldots, x_k)$.
\end{theorem}

\begin{proof}
Let ${ \bf x_k}$ denote $(x_1, \ldots, x_k)$. 
It is enough to prove that the transition matrix from row-strict quasisymmetric Schur functions to fundamental quasisymmetric functions is upper uni-triangular with respect to a certain order on compositions.  Order the compositions indexing the row-strict quasisymmetric Schur functions by the revlex order, $\ge_r$, where $\alpha \ge_r \beta$ if and only if either $\shape(\alpha)  > \shape(\beta)$ or $\shape(\alpha) = \shape(\beta)$ and $\alpha \ge_{lex} \beta$.  Order the compositions indexing the fundamental quasisymmetric functions by the revlex order on their complementary compositions.

Fix a positive integer $n$ and a composition $\alpha=(\alpha_1, \alpha_2, \hdots , \alpha_{\ell(\alpha)})$ of $n$.  Consider a summand $F_{\beta}({\bf x_k})$ appearing in $\rowq_{\alpha}({\bf x_k})$ and consider the composition $\tilde{\beta}$ complementary to $\beta$.  Since $k \ge n$ we know that $F_{\beta}({\bf x_k})$ is nonzero.  We claim that $\shape(\tilde{\beta}) \le \shape(\alpha)$ and, moreover, if $\shape(\tilde{\beta}) = \shape(\alpha)$ then $\tilde{\beta} = \alpha$.

Since $F_{\beta}({\bf x_k})$ appears in $\rowq_{\alpha}({\bf x_k})$, there must exist a standard reverse row-strict composition tableau $F$ of shape $\alpha$ with transpose descent set $D=\{b_1, b_2, \hdots b_k\}$ such that $\bm\beta(D)=\beta$.  Each entry $i$ in $D$ must appear in $F$ strictly to the right of $i+1$.  Therefore the consecutive entries in $D$ appear in $F$ as horizontal strips.  Each such horizontal strip corresponds to a part of $\tilde{\beta}$ since $\tilde{\beta}$ arises from the complement of $D$.  This implies that the sum of the largest $j$ parts of $\tilde{\beta}$ must be less than or equal to the sum of the largest $j$ parts of $\alpha$.  Therefore $\shape(\tilde{\beta}) \le \shape(\alpha)$.

Assume that $\shape(\tilde{\beta}) = \shape(\alpha)$.  Since the length of $\tilde{\beta}$ is equal to the length of $\alpha$, the horizontal strips corresponding to the parts of $\tilde{\beta}$ must be left-justified in $F$ and hence the strip corresponding to $\beta_i$ must begin in the $i^{th}$ row of $\alpha$.  Therefore $\tilde{\beta}_1 \ge \alpha_1$ since the row entries of $F$ must be strictly decreasing from left to right.  Iterating this argument implies that $\tilde{\beta}_i = \alpha_i$ for all $i$, and hence $\tilde{\beta} = \alpha$ as claimed.  Therefore the transition matrix is upper uni-triangular and the proof is complete.
\end{proof}

See Figure {\ref{ex:matrixa}} for an example of the transition matrix from 
the row-strict quasisymmetric Schur functions to fundamental quasisymmetric functions for $n=4$. Figure {\ref{ex:matrixb}} gives the transition matrix from 
the quasisymmetric Schur functions to the fundamental quasisymmetric functions for $n=4$.

\begin{figure}
\begin{center}
\subfigure[The transition matrix from row-strict quasisymmetric Schur functions to fundamental quasisymmetric Schur functions]{
\begin{tabular}{l || c | c | c | c | c | c | c | c |}
 $\rowq_{\alpha}\backslash F_\beta$ & 1111 & 112 & 211 & 121 & 13 & 22 & 31 & 4 \\
\hline
\hline
4 & 1 & 0 & 0 & 0 & 0 & 0 & 0 & 0 \\
\hline
31 & 0 & 1 & 0 & 0 & 0 & 0 & 0 & 0 \\
\hline
13 & 0 & 0 & 1 & 1 & 0 & 0 & 0 & 0 \\
\hline
22 & 0 & 0 & 0 & 1 & 0 & 1 & 0 & 0 \\
\hline
211 & 0 & 0 & 0 & 0 & 1 & 0 & 0 & 0 \\
\hline
121 & 0 & 0 & 0 & 0 & 0 & 1 & 0 & 0 \\
\hline
112 & 0 & 0 & 0 & 0 & 0 & 0 & 1 & 0 \\
\hline
1111 & 0 & 0 & 0 & 0 & 0 & 0 & 0 & 1 \\
\hline
\end{tabular}
\label{ex:matrixa}
}\\
\subfigure[The transition matrix from quasisymmetric Schur functions to fundamental quasisymmetric Schur functions]{
\begin{tabular}{l || c | c | c | c | c | c | c | c |}
 $\qschur_{\alpha}\backslash F_\beta$& 1111 & 112 & 211 & 121 & 13 & 22 & 31 & 4 \\
\hline
\hline
4 & 0 & 0 & 0 & 0 & 0 & 0 & 0 & 1 \\
\hline
31 & 0 & 0 & 0 & 0 & 0 & 0 & 1 & 0 \\
\hline
13 & 0 & 0 & 0 & 0 & 1 & 1 & 0 & 0 \\
\hline
22 & 0 & 0 & 0 & 1 & 0 & 1 & 0 & 0 \\
\hline
211 & 0 & 0 & 1 & 0 & 0 & 0 & 0 & 0 \\
\hline
121 & 0 & 0 & 0 & 1 & 0 & 1 & 0 & 0 \\
\hline
112 & 0 & 1 & 0 & 0 & 0 & 0 & 1 & 0 \\
\hline
1111 & 1 & 0 & 0 & 0 & 0 & 0 & 0 & 0 \\
\hline
\end{tabular}
\label{ex:matrixb}
}
\end{center}
\end{figure}


\section{A linear endomorphism of QSym}{\label{omega}}

An algebra endomorphism $\omega: \Lambda \rightarrow \Lambda$ on symmetric functions is defined by $\omega(e_n)=h_n, \; n \ge 1$.  The definitions of $e_{\lambda}$ and $h_{\lambda}$ imply that $\omega(e_{\lambda}) = h_{\lambda}$ since $\omega$ preserves multiplication.  The endomorphism is an involution and $\omega(s_{\lambda}) = s_{\lambda'}$, where $\lambda'$ is the transpose of the partition $\lambda$, often referred to as the conjugate partition~\cite{Hal59, Hal88, Sta71}.

\subsection{An extension of the $\omega$ automorphism of symmetric functions}

\begin{theorem}{\label{omegathm}}
The $\omega$ operator maps quasisymmetric Schur functions to row-strict quasisymmetric Schur functions.  That is,
$$\omega(\qschur_{\alpha}(x_1, \hdots , x_n)) =\rowq_{\alpha}(x_n, \hdots, x_1).$$
\end{theorem}

\begin{proof}
Proposition \ref{qschur2fund} states that $\qschur_{\alpha}$ can be written as a positive sum of fundamental quasisymmetric functions in the following way.

\begin{equation}{\label{qschur}}
\qschur_{\alpha}(x_1, \hdots , x_n) = \sum_{\beta} d_{\alpha \beta} F_{\beta}(x_1, \hdots , x_n),
\end{equation}
where $d_{\alpha \beta}$ is equal to the number of standard column-strict composition tableaux $T$ of shape $\alpha$ such that $\bm\beta (D(T)) = \beta$.  Apply $\omega$ to Equation \ref{qschur} to get
\begin{eqnarray*}
\omega(\qschur_{\alpha}(x_1, \hdots , x_n))  & = & \sum_{\beta} d_{\alpha \beta} \omega(F_{\beta}(x_1, \hdots , x_n)) \\
 & = & \sum_{\beta} d_{\alpha \beta} F_{rev(\tilde{\beta})}(x_1, \hdots , x_n) \\
& = & \sum_{\beta} d_{\alpha \beta} F_{\tilde{\beta}}(x_n, \hdots , x_1) \\
& = & \rowq_{\alpha}(x_n, \hdots , x_1)
\end{eqnarray*}
The second equality above results directly from the definition of the $\omega$ operator on the fundamental quasisymmetric functions~\cite{Ehr96},~\cite{Ges84},~\cite{MalReu95} and the last equality results from the fact that transposing a standard Young tableau replaces the descent set $D$ with the complementary set $D^c$.
\end{proof}

\subsection{A notion of conjugation for compositions}

The $\omega$ operator applied to a Schur function $s_{\lambda}$ produces the Schur function $s_{\lambda'}$ indexed by the conjugate partition to $\lambda$.  Since the quasisymmetric Schur functions and the row-strict quasisymmetric Schur functions are generated by sums of monomials arising from fillings of composition diagrams, it is natural to seek a conjugation-like operation on composition diagrams which preserves the transpose of the underlying partition.  

This idea at first seems to be too much to ask since there cannot be a bijection between compositions that rearrange to a given partition $\lambda$ and compositions that rearrange to its conjugate partition $\lambda'$ for 
all $\lambda$.  Consider for example the partition $\lambda=(2,1,1)$.  There are three compositions ($(2,1,1), (1,2,1)$ and $(1,1,2)$) which rearrange $\lambda$ but only two compositions ($(3,1)$ and $(1,3)$) which rearrange $\lambda'=(3,1)$.  However, the $\omega$ operator can be used to collect the compositions which rearrange a given partition so that there exists a bijection between these collections and the collections corresponding to the conjugate partition.  In particular, note that since $\omega$ sends a quasisymmetric Schur function to a row-strict quasisymmetric Schur function, a method for writing the quasisymmetric Schur functions in terms of the row-strict quasisymmetric Schur functions and vice versa would allow us to interpret the indexing compositions in much the same way as we interpret the indexing partitions for Schur functions and their images under $\omega$.

Recall that the Schur functions can be expanded into sums of either quasisymmetric Schur functions or dual quasisymmetric Schur functions as follows:  $$\sum_{\shape(\alpha)=\lambda} \qschur_{\alpha} = s_{\lambda} = \sum_{\shape(\beta)=\lambda'} \rowq_{\beta}.$$  This suggests that there must be a weight-preserving bijection between column-strict composition tableaux and row-strict composition tableaux which transposes the shape of the underlying partition.

Given an arbitrary column-strict composition tableau $F$, choose the largest entry in each column of $F$ and construct the leftmost column of $\phi(F)$ by placing these entries (the collection $C_1$) in weakly increasing order from top to bottom.  Then choose the second-largest entry from each column of $F$ (ignoring empty cells) and insert this collection ($C_2$) into the new diagram by the following procedure.  Place the largest entry $e_i$ into the highest position in the second column so that the entry immediately to the left of $e_i$ is strictly greater than $e_i$.  Repeat this insertion with the next largest entry, considering only the unoccupied positions.  Continue this procedure with the remainder of the entries in this collection from largest to smallest until the entire collection has been inserted.  Then repeat the procedure for the third largest entry in each column ($C_3$) of $F$ and continue inserting collections of entries until all entries of $F$ have been inserted.  The resulting diagram is $\phi(F)$.  See Figure {\ref{conj}} for an example.

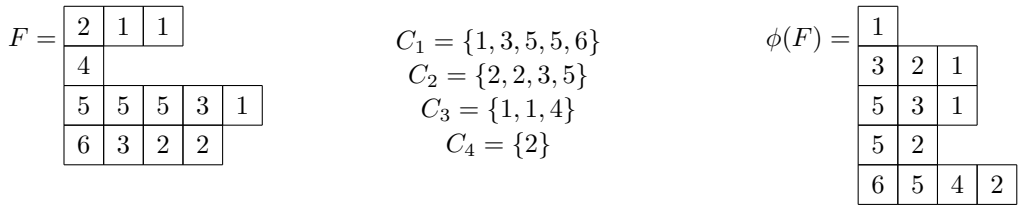
\begin{figure}{\label{conj}}
\begin{center}
\begin{picture}(400,100)
\put(0,60){ $F=\tableau{2 & 1 & 1 \\ 4 \\ 5 & 5 & 5 & 3 & 1 \\ 6 & 3 & 2 & 2 }$} 
\put(150,20){\shortstack[b]{$C_1=\{1,3,5,5,6\}$ \\ $C_2=\{2,2,3,5\}$ \\ $C_3=\{1,1,4\}$ \\ $C_4=\{2\}$ }}
\put(290,60){$\phi(F)=\tableau{1 \\ 3 & 2 & 1 \\ 5 & 3 & 1 \\ 5 & 2 \\ 6 & 5 & 4 & 2}$}
\end{picture}
\caption{The map $\phi$ from a $CSCT$ to an $RSCT$}
\end{center}
\end{figure}

\begin{proposition}
The map $\phi$ is a weight-preserving bijection between column-strict composition tableaux which rearrange a given partition and row-strict composition tableaux which rearrange the conjugate partition.
\end{proposition}

\begin{proof}
We must first prove that the map $\phi$ is well-defined.  Consider the $i^{th}$ largest entry in an insertion collection $C_j$.  This entry is strictly smaller than the entry from its column in $C_{j-1}$.  It is also weakly smaller than the $i-1$ largest entries in the collection $C_j$ and hence strictly smaller than at least $i$ of the entries in the collection $C_{j-1}$.  Therefore there exists at least one position in which this entry may be placed and the map is well-defined.  
The fact that $\phi$ is weight-preserving and conjugates the underlying partition is immediate from the definition of $\phi$.  

To see that $\phi(F)$ is indeed a row-strict composition tableau, first note that conditions (1) and (2) are satisfied by construction.  Next consider a pair of entries $e_1$ and $e_2$ in the same column of $\phi(F)$ such that $e_1 < e_2$ and $e_1$ is in a higher row of $\phi(F)$ than $e_2$.  Then $e_2$ was inserted before $e_1$ and hence the position containing $e_1$ was available during the insertion of $e_2$.  Since $e_2$ was not placed in this position, the entry in the cell immediately to the left of $e_1$ must be weakly smaller than $e_2$.  Therefore condition (3) is satisfied and $\phi(F)$ is a row-strict composition tableau.

Let $\rho_r^{-1}$ be the map from reverse row-strict tableaux to row-strict composition tableaux given in this paper, let $\rho_c^{-1}$ be the analogous map from reverse column-strict tableaux to column-strict composition tableaux given in~\cite{HLMvW09}, and let $\tau$ be the transposition map that sends a reverse semi-standard Young tableau to its conjugate by reflecting across the main diagonal.  Then $\phi = \rho_r^{-1} \circ \tau \circ \rho_c^{-1}$.  The map $\phi$ is therefore a bijection with inverse $\phi^{-1}$ given by $\phi^{-1} = \rho_c \circ \tau \circ \rho_r$.
\end{proof}

We note that $\phi$ does not always preserve compositions shapes. That is, 
it is not necessarily the case if $T_1$ and $T_2$ are column strict 
composition tableaux of the same shape, the $\phi(T_1)$ and 
$\phi(T_2)$ have the same shape. See Figure {\ref{conjugate}} for an example. 
Indeed, its is not possible to define such a shape preserving map 
since the number of compositions that rearrange to $\lambda$ is not 
always the same as the number of compositions that rearrange to 
$\lambda'$.

\begin{figure}
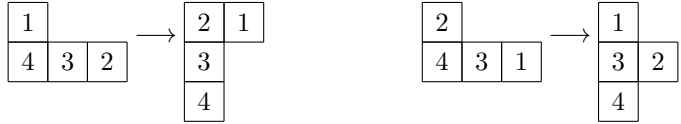
{\label{conjugate}}
$$\tableau{1 \\ 4 & 3 & 2} \longrightarrow \tableau{2 & 1 \\ 3 \\ 4} \qquad \qquad \qquad \tableau{2 \\ 4 & 3 & 1} \longrightarrow \tableau{1 \\ 3 & 2 \\ 4}$$
\caption{The compositions $(2,1,1)$ and $(1,2,1)$ are conjugate to $(1,3)$}
\end{figure}

\section{An extension of dual Schensted insertion}{\label{Dual}}

Schensted insertion provides a method for inserting an arbitrary positive integer into an arbitrary column-strict tableau.  This insertion process forms the foundation for the well-known Robinson-Schensted-Knuth (RSK) algorithm which produces a bijection between matrices with non-negative integer coefficients and pairs of reverse column-strict tableaux of the same shape.  Mason \cite{Mas08} 
gave an extension of RSK algorithm to insert an integer into a 
column-strict composition tableau and 
Haglund, Luoto, Mason and van Willigenburg \cite{HLMvW10} used 
to insertion procedure to prove the refinement of the 
Littlewood-Richardson rule which gives a combinatorial 
interpretation to the coefficients that arise in the expansion of  
the product of a Schur function times a quasisymmetric Schur function  
into a positive sum of quasisymmetric Schur functions.

In this section, we define an extension of dual Schensted insertion. 
This extension is used by Ferreira~\cite{Fer11} to give another refinement of the Littlewood-Richard rule which provides a combinatorial 
interpretation for the coefficients that arise in the expansion  
the product of a Schur function times a row-strict 
quasisymmetric Schur function 
into a positive sum of quasisymmetric Schur functions.

Dual Schensted insertion inserts an arbitrary positive integer into a reverse row-strict tableau by bumping entries from row to row.  That is, given a reverse row-strict tableau $T$ and a positive integer $x$, first set $R$ equal to the first row of $T$.  Let $y$ be the largest entry in $R$ less than or equal to $x$.  Replace $y$ by $x$ in $R$.  Set $x=y$ and set $R$ equal to the next row down and repeat.  If there is no such entry $y$ which is less than or equal to $x$, place $x$ at the end of row $R$ and stop.   The resulting figure is denoted $T \leftarrow x$.  A more detailed exposition of Schensted insertion and its variations can be found in \cite{Knu73} or \cite{Sta99}.  We need only the following Lemma concerning the insertion path.

\begin{lemma}{\label{rowinsertion}}
The insertion path consisting of all cells affected by the insertion process moves weakly to the left as the cells are listed from top to bottom.
\end{lemma}

The following analogue of dual Schensted insertion provides a method for inserting a new cell into a RSCT.  Given an arbitrary RSCT $F$, let $\rea(F)$ be the {\it reading word} for $F$ given by reading the entries of $F$ by column from right to left, reading the columns from top to bottom.  This ordering of the cells is called the {\it reading order} on the cells of $F$.  The {\it modified reading word $\rea(\tilde{F})$} for $F$ is given by appending a cell containing the entry $0$ after the rightmost cell in each row to obtain $\tilde{F}$ and then recording the entries of $\tilde{F}$ in reading order.  (See Figure \ref{read} for an example.)

\begin{figure}
\begin{center}
\begin{picture}(200,80)
\put(0,60){$F=\tableau{2 & 1 \\ 2 \\ 4 & 3 & 2 \\ 4 & 2 \\ 5 & 2}$}
\put(150,60){$\tilde{F} = \tableau{2 & 1 & 0 \\ 2 & 0 \\ 4 & 3 & 2 & 0 \\ 4 & 2 & 0 \\ 5 & 2 & 0}$}
\end{picture}
\caption{$\rea(\tilde{F}) =00200  10  3  2  2  2  2 445$ }
{\label{read}}
\end{center}
\end{figure}

To insert an arbitrary positive integer $x$ into $F$, scan $\rea(\tilde{F})$ to find the first entry $y$ less than or equal to $x$ such that the entry immediately to the left of $y$ in $\tilde{F}$ is greater than $x$.  (If such a $y$ does not exist or if $y$ is in the leftmost column, place $x$ after the last entry smaller than or equal to $x$ in the leftmost column, shifting the lower rows down by one, and stop.)  If $y=0$, then replace $y$ by $x$ and stop.  Otherwise, replace $y$ by $x$ in 
which case we say that $x$ {\em bumps} $y$ and repeat the procedure using $y$ instead of $x$ and considering only the portion of $\rea(\tilde{F})$ appearing after $y$.  The resulting figure is denoted $F \leftharpoondown x$.  See Figure \ref{schensted} for an example. Note that $x$ followed by  all the elements which 
were bumped in the insertion procedure (read in reading order) must form a weakly decreasing sequence. 

\begin{figure}
\begin{center}
\begin{picture}(470,120)
\put(0,100){$3 \rightarrow \tableau{2 & 1 & 0 \\ 2 & 0 \\ 4 & 3 & 2 & 0 \\ 4 & 2 & 0 \\ 5 & 2 & 0}$}
\put(20,15){$3 \rightarrow$}
\put(0,0){$00200  10  {\bf 3}  2  2  2  2 445$}
\put(120,100){$\tableau{2 & 1 & 0 \\ 2 & 0 \\ 4 & {\bf 3} & 2 & 0 \\ 4 & 2 & 0 \\ 5 & 2 & 0}$}
\put(120,0){$3 \rightarrow {\bf 2} 2 2 2 4 4 5$}
\put(220,100){$\tableau{2 & 1 & 0 \\ 2 & 0 \\ 4 & 3 & 2 & 0 \\ 4 & {\bf 3} & 0 \\ 5 & 2 & 0}$}
\put(220,0){$2 \rightarrow  2 2 2 4 4 5$}
\put(320,100){$\tableau{2 & 1 & 0 \\ 2 & 0 \\ 4 & 3 & 2 & 0 \\ 4 & {3} & 0 \\ 5 & {\bf 2} & 0}$}
\put(320,0){$2 \rightarrow  2 2 4 4 5$}
\put(420,100){$\tableau{2 & 1  \\ 2  \\  {\bf 2} \\ 4 & 3 & 2  \\ 4 & 3 \\ 5 & 2 }$}
\put(420,0){$F \leftharpoondown 3$}
\end{picture}
\caption{The insertion procedure $F \leftharpoondown 3$}
\end{center}
{\label{schensted}}
\end{figure}

\begin{lemma}{\label{insRCT}}
The insertion procedure $F \leftharpoondown x$ produces an RSCT.
\end{lemma}

\begin{proof}
The insertion procedure $F \leftharpoondown x$ satisfies the first and second RSCT conditions by construction.  To see that $F'=F \leftharpoondown x$ satisfies the triple condition, consider the supplemented tableau $\hat{F'}$ and select two entries $\hat{F'}(i,k), \hat{F'}(j,k)$ in the same column of $\hat{F'}$ such that $i < j$ and $\hat{F'}(j,k) > \hat{F'}(i,k)$.  Assume that $\hat{F'}(j,k) < \hat{F'}(i,k-1)$, for otherwise the triple condition is satisfied.  Set $a:=\hat{F'}(i,k), b:=\hat{F'}(j,k)$, and $c:=\hat{F'}(i,k-1)$.  The set of entries affected by the insertion procedure form a weakly decreasing sequence when read in reading order, which implies that at most one of the entries $a,b,c$ is affected by the insertion procedure since $a<b<c$.  

Now it cannot be that
$a$ bumped $F(i,k)$ since then column $k$ must have been 
part of $\hat{F}$ and it would be the case that 
$\hat{F}(i,k) \leq a < \hat{F}(j,k) =b < 
\hat{F}(i,k-1) = c$ 
which would violate the row-strict triple condition in $F$. 

We also claim that $b$ cannot be involved in the insertion 
procedure.  To see this, assume that $b$ is involved and argue by contradiction.  In this case, if column $k$ is 
not part of $\hat{F}$, then $b$ must have been 
the terminal cell in the insertion procedure and all 
the other elements in column $k$ must be 0. Thus, in particular, 
$a =0$.  In such a situation, we must have tried to insert 
$b$ in cell $(i,k)$ and the only reason that we did not 
place $b$ in cell $(i,k)$ must have been that $b \leq F(i,k-1) = c$ 
which violates our assumption. Thus it must have been 
that case that column $k$ is in $\hat{F}$. Then 
we claim that $b$ could not have bumped $\tilde{F}(j,k)$. 
Otherwise 
consider the element $x'$ which was being inserted by the time 
that we reached cell $(i,k)$ in the reading order. Since 
either $x' =x$ or $x'$ was bumped, we know that 
$x' \geq b > a$.  The only reason that 
$x'$ would not have bumped $a$ is if $x' \geq c$.  In particular, $x' \neq b$. 
Now suppose  that 
the set of elements that were bumped in column $k$ between 
row $i$ and row $j$ were $i < s_1 < \cdots < s_t < j$. Then it 
must be the case that $F(k,s_t) =b$. Thus let $r$ be the least 
element $p$ such that $F(s_p,k) < c$.  Now it cannot be 
that $F(s_r,k) > a$ since otherwise we would have 
$c = \hat{F}(i,k-1) > \hat{F}(s_r,k) > a =\hat{F}(i,k)$ which would violate 
the row-strict triple condition in $F$.  But then we would 
have that $a \geq F(s_r,k) \geq F(s_{r+1},k) \geq \cdots \geq F(s_t,k) = b$ 
since elements the elements which are bumped from a weakly decreasing 
sequence when read in reading order. This also violates our 
assumption so $b$ must not have been involved in the bumping 
process.

Thus we must assume that $c$ bumps $c' =F(i,k-1)$. Note that in 
this case we have $\hat{F}(i,k) =a < \hat{F}(j,k) =b$ so 
that we must have $c' \leq b < c$ by the row-strict triple condition 
for $F$. Let $d=F(j,k-1)$ and $e = F(i,k-2)$. 
The cells are arranged in $F$ as depicted below.
$$\tableau{e & c' & a \\ \\ & d & b}$$
Now $d > b \geq c'$ so that the row-strict triple condition 
for $F$ also implies that $d \geq e$. But since $c$ bumps $c'$, 
it must be the case that $e > c$. Thus we know that 
$d \geq e > c > b \ge c' > a$.  Now 
consider the question of where $c$ came from in the bumping 
process. It cannot be that $c =x$ or $c$ came from a cell 
before $(j,k)$ since otherwise we would insert 
$c$ into $(j,k)$ since we have shown 
that $d > c > b$.  Thus $c$ must have come from a cell 
after $(j,k)$. Now $c$ could not have come from a cell 
$(s,k)$ with $s > j$ since then we would have 
that $F(j,k) = b < F(s,k) = c < F(j,k-1) =d$ which would 
violate the row-strict triple condition in $F$.  
Thus $c$ must have been bumped 
from column $k-1$.  Now let $(t_1,k-1), \ldots, (t_s,k-1) = (i,k-1)$ 
be the cells in column $k-1$ whose elements were bumped 
in the insertion process weakly before bumping $c'$, where $t_1 < \cdots < t_s$.  Then 
we know that $s > 1$ and 
$$F(t_1,k-1) \geq \cdots \geq F(t_{s-1},k-1) = c > F(t_s,k-1) =c'.$$
Moreover since $F(t_p,k-1)$ bumps $F(t_{p+1},k-1)$ for 
$1 \leq p < s$, it must also be the case 
that $F(t_p,k-1) < F(t_{p+1},k-2)$ for $1 \leq p < s$.  
In particular, we know that 
$d = F(j,k-1) \geq e = F(t_s,k-2) > F(t_{s-1},k-1)$ so 
that the row-strict triple condition for $F$ implies 
that $d \geq F(t_{s-1},k-2)$.  But then 
$d \geq F(t_{s-1},k-2)> F(t_{s-2},k-1)$ so that 
row-strict triple condition for $F$ implies 
that $d \geq F(t_{s-2},k-2)$. Continuing on in 
this way, we can conclude that $d \geq F(t_1,k-2) > F(t_1,k-1)$. 
Now consider which element 
$f$ could have bumped $F(t_1,k-1)$.  This element $f$ must have come 
from column $k$ or higher. Now we know 
that 
$$d \geq F(t_1,k-2) > f \geq F(t_1,k-1) \geq F(t_s,k-1) = c > b.$$ 
Thus 
$f$ could not be $x$ or come from a cell before $(j,k)$ 
since otherwise we would try to insert $f$ in cell $(j,k)$ and 
since $d > f > b$, $f$ would bump $b$.  Thus 
$f = F(r,k)$ where $r >j$. But then 
we would have that $d= F(j,k-1)> F(r,k) =f > F(j,k) =b$ which 
would violate the row-strict triple condition for $F$.  
Thus we can conclude that $c$ also was not involved in 
the insertion procedure which would mean 
that $a$, $b$, and $c$ would violate the row-strict triple 
condition for $F$.  Hence, there can be no violation 
of the row-strict triple condition in $F'$, so 
$F'$ is a RSCT. 
\end{proof}

Our extension of the dual Schensted algorithm has a number of 
nice properties in addition to Ferreira's Littelwood-Richardson result \cite{Fer11}.  In particular, it is straightforward to prove 
that the following theorems.

\begin{theorem}{\label{insertioncommutes}}
The insertion procedure on RSCT commutes with the reverse row insertion in the sense that $\rho(T \leftarrow x) = \rho(T) \leftharpoondown x$.
\end{theorem}

See Figure \ref{fig:counter} for an example of Theorem \ref{insertioncommutes}.

\begin{figure}
\includegraphics{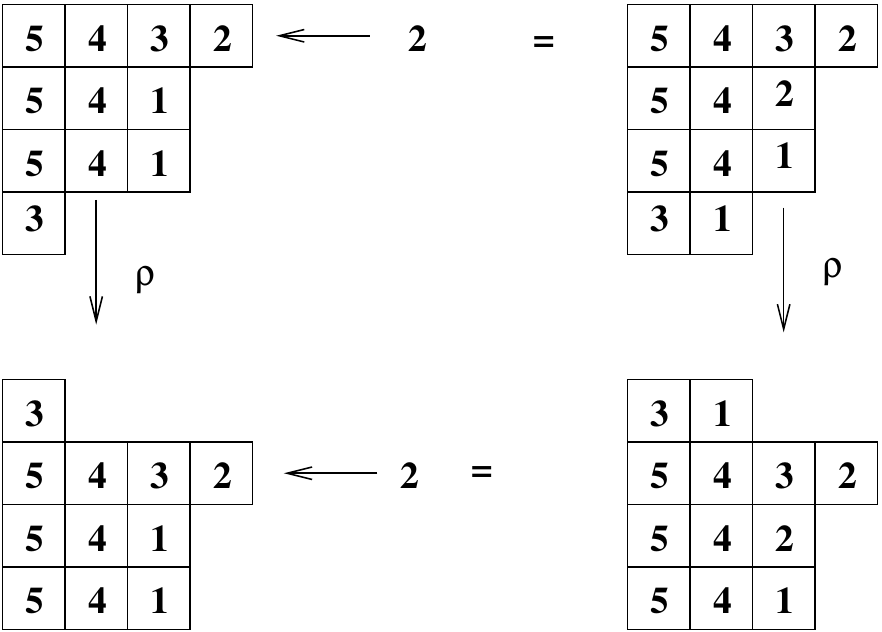}
\caption{An example of Theorem \ref{insertioncommutes}}
\end{figure}

\begin{theorem}
There exists a bijection between $\mathbb{N}$-matrices with finite support and pairs $(F,G)$ of row-strict composition tableaux which rearrange the same partition.
\end{theorem}





\bibliographystyle{plain}
\bibliography{hmrbib}

\end{document}